\RequirePackage[l2tabu, orthodox]{nag}
\documentclass[a4paper,final,11pt,leqno
]{article}

\newcommand{\mytitle}{Double-normal pairs in the plane\\ and on the sphere}

\usepackage{overpic,pgf,tikz,bm,enumerate}
\usetikzlibrary{arrows}

\usepackage[final]{microtype}
\makeatletter
\def\MT@register@subst@font{\MT@exp@one@n\MT@in@clist\font@name\MT@font@list
   \ifMT@inlist@\else\xdef\MT@font@list{\MT@font@list\font@name,}\fi}
\makeatother 

\usepackage[small]{titlesec}

\usepackage[notref,notcite]{showkeys}

\usepackage{datetime}
\longdate

\usepackage{mparhack,fixltx2e}

\usepackage{amsmath, amssymb, amsfonts, amsthm, mathtools}
\usepackage[all,warning]{onlyamsmath}
\usepackage{fixmath}
\mathtoolsset{centercolon}

\usepackage[latin1]{inputenc}
\usepackage[T1]{fontenc}

\usepackage[draft=false,pdftex]{hyperref}
\hypersetup{
    bookmarks=true,         
    unicode=false,          
    pdftoolbar=true,        
    pdfmenubar=true,        
    pdffitwindow=false,     
    pdfstartview={FitH},    
    pdftitle={\mytitle},    
    pdfauthor={Janos Pach and Konrad J. Swanepoel},     
    pdfnewwindow=false,      
    colorlinks=true,       
    linkcolor=black,          
    citecolor=black,        
    filecolor=black,      
    urlcolor=black           
}

\theoremstyle{plain}
\newtheorem{theorem}{Theorem}
\newtheorem{lemma}[theorem]{Lemma}

\newtheorem{claim}[theorem]{Lemma}

\theoremstyle{definition}

\usepackage{xspace}

\newcommand{\setbuilder}[2]{\left\{#1\;\colon\,#2\right\}}
\newcommand{\set}[1]{\left\{#1\right\}}

\newcommand{\myangle}{\angle}
\newcommand{\length}[1]{\lvert#1\rvert}
\newcommand{\card}[1]{\left\lvert#1\right\rvert}
\newcommand{\numbersystem}[1]{\mathbb{#1}}

\newcommand{\bR}{\numbersystem{R}}
\newcommand{\bS}{\numbersystem{S}}

\DeclareMathOperator{\conv}{conv}
\DeclareMathOperator{\diam}{diam}
\newcommand{\vect}[1]{\bm{#1}}
\newcommand{\va}{\vect{a}}
\newcommand{\vb}{\vect{b}}
\newcommand{\vc}{\vect{c}}
\newcommand{\vd}{\vect{d}}
\newcommand{\ve}{\vect{e}}

\newcommand{\vo}{\vect{o}}
\newcommand{\vp}{\vect{p}}
\newcommand{\vq}{\vect{q}}

\newcommand{\vs}{\vect{s}}

\newcommand{\vv}{\vect{v}}

\newcommand{\vx}{\vect{x}}
\newcommand{\vy}{\vect{y}}
\newcommand{\vz}{\vect{z}}

\newcommand{\define}[1]{\emph{#1}}

\title{\mytitle}
\author{J\'anos Pach\thanks{Research partially supported by Swiss National Science Foundation Grants 200021-137574 and 200020-144531, by Hungarian Science Foundation Grant OTKA NN 102029 under the EuroGIGA programs ComPoSe and GraDR, and by NSF grant CCF-08-30272.}\\
EPFL Lausanne and \\
R\'enyi Institute, Budapest\\
\href{mailto:pach@cims.nyu.edu}{\texttt{pach@cims.nyu.edu}}
\and 
Konrad J.\ Swanepoel\\
Department of Mathematics,\\ London School of Economics and Political Science,\\ Houghton Street, London WC2A 2AE, United Kingdom\\
\href{mailto:k.swanepoel@lse.ac.uk}{\texttt{k.swanepoel@lse.ac.uk}}
}
\date{}

\begin{document}
\maketitle
\begin{abstract}
A \define{double-normal pair} of a finite set $S$ of points from Euclidean space is a pair of points $\set{\vp,\vq}$ from $S$ such that $S$ lies in the closed strip bounded by the hyperplanes through $\vp$ and $\vq$ that are perpendicular to~$\vp\vq$.
A double-normal pair $\vp\vq$ is \define{strict} if $S\setminus\set{\vp,\vq}$ lies in the open strip.
We answer a question of Martini and Soltan (2006) by showing that 
a set of $n\geq 3$ points in the plane has at most $3\lfloor n/2\rfloor$ double-normal pairs.
This bound is sharp for each $n\geq 3$.

In a companion paper, we have asymptotically determined this maximum for points in $\bR^3$.
Here we show that if the set lies on some $2$-sphere, it has at most
$17n/4 - 6$ double-normal pairs.
This bound is attained for infinitely many values of $n$.

We also establish tight bounds for the maximum number of strict double-normal pairs in a set of $n$ points in the plane and on the sphere.
\end{abstract}


\section{Introduction}
Let $V$ be a set of $n$ points in Euclidean space.
A \define{double-normal pair} of $V$ is a pair of points $\set{\vp,\vq}$ in $V$ such that $V$ lies in the closed strip bounded by the hyperplanes $H_{\vp}$ and $H_{\vq}$ through $\vp$ and $\vq$, respectively, that are perpendicular to $\vp\vq$.
A double-normal pair $\vp\vq$ is \define{strict} if $V\setminus\set{\vp,\vq}$ is disjoint from $H_{\vp}$ and $H_{\vq}$.
Define the \define{double-normal graph} of $V$ as the graph on the vertex set $V$ in which two vertices $p$ and $q$ are joined by an edge if and only if $\{p,q\}$ is a double-normal pair.
The number of edges of this graph, that is, the number of double-normal pairs induced by $V$, is denoted by $N(V)$.

We define the \define{strict double-normal graph} of $V$ analogously and denote its number of edges by $N'(V)$.

Martini and Soltan \cite[Problems~3 and 4]{martini-soltan-2005} initiated the investigation of the maximum number of double-normal pairs and strict double-normal pairs of a set of $n$ points in $\bR^d$.
Define
\[ N_d(n) := \max_{\substack{V\subset \bR^d\\ \card{V}=n}} N(V)\]
and
\[ N_d'(n) := \max_{\substack{V\subset \bR^d\\ \card{V}=n}} N'(V).\]
Clearly, we have $N(V)\geq N'(V)$, hence, $N_d(n)\geq N_d'(n)$.

A lower bound to $N_d'(n)$ is provided by the maximum number of diameter pairs that can occur in a set of $n$ points.
A \define{diameter pair} of $S$ is a pair of points $\{\vp,\vq\}$ in $S$ such that $\length{\vp\vq}=\diam(S)$.
Let $M_d(n)$ denote the maximum number of diameter pairs of a set of $n$ points in $\bR^d$.
Since a diameter pair of $S$ is also a strict double-normal pair of $S$, $M_d(n)\leq N_d'(n)$.
It is well-known that $M_2(n)=n$ for $n\geq 3$ \cite{Erdos46}
and $M_3(n)=2n-2$ for $n\geq 4$ \cite{Grmax, Heppesmax, Strasmax}, thus 
giving $N_2(n)\geq N_2'(n)\geq n$ and $N_3(n)\geq N_3'(n)\geq 2n-2$.

Since any two strict double-normal pairs without common endpoints in the plane have to cross, it follows from the same well-known proof due to Perles that gives $M_2(n)\leq n$ \cite[Theorem~9]{Pach}, that a set of $n$ points in the plane has at most $n$ strict double-normal pairs, that is, $N'_2(n)\leq n$.
Thus, the exact value $N_2'(n)=n$ for $n\geq 3$ follows from the above results.
Our next theorem states that $N_2(n)=3\lfloor n/2\rfloor$.

\begin{theorem}\label{thm:2d}
Given a finite set $V$ of at least~$3$ points in the plane,
the number of double-normal pairs in $V$ satisfies 
\[N(V)\leq 3\left\lfloor\frac{\card{V}}{2}\right\rfloor\text{.}\]
This bound can be attained for all $\card{V}\geq 3$.
If $\card{V}$ is even and $N(V)=3\card{V}/2$, then $V$ lies on a circle and is symmetric with respect to the centre of the circle.
\end{theorem}
For even values of $n=\card{V}$, the sharpness of the bound in Theorem~\ref{thm:2d} is shown by the vertex set of a regular $n$-gon (Fig.~\ref{octagon}).
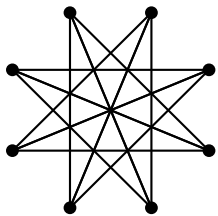
\begin{figure}
\centering
\begin{tikzpicture}[scale=0.7, thick]
\foreach \theta in {0,45,90,...,315} {
    \draw[fill] (\theta+22.5:2cm) circle (1mm);
}
\foreach \theta in {0,45,90,...,315} {
    \draw (\theta+22.5:2cm) -- (\theta+22.5+180:2cm);
}
\foreach \theta in {0,45,90,...,315} {
    \draw (\theta+22.5:2cm) -- (\theta+22.5+180+45:2cm);
}
\end{tikzpicture}
\caption{$12$ double-normal pairs among the vertices of a regular octagon}\label{octagon}
\end{figure}
To obtain an extremal example with an odd number of points, simply add any other point in the interior or on the boundary of the $n$-gon.
For odd $n$, there are other, combinatorially distinct, examples, such as the one in Fig.~\ref{7points}.
\begin{figure}
\centering
\begin{tikzpicture}[scale=0.85, thick]
\coordinate (a) at (-0.7,0);
\coordinate (b) at (0.3,0);
\coordinate (c) at (1,-1);
\coordinate (d) at (0.3,-3.1);
\coordinate (e) at (-0.7,-3.1);
\coordinate (f) at (-1.3,-3);
\coordinate (g) at (-2,-2);
\draw (a)--(e);
\draw (a)--(d);
\draw (b)--(d);
\draw (b)--(e);
\draw (b)--(f);
\draw (b)--(g);
\draw (c)--(f);
\draw (c)--(g);
\draw (d)--(g);
\draw[ultra thick,dotted] (a)--(b);
\draw[ultra thick,dotted] (d)--(e);
\draw[ultra thick,dotted] (b)--(c);
\draw[ultra thick,dotted] (g)--(f);
\foreach \x in {a,b,c,d,e,f,g} {
    \draw[fill=black] (\x) circle (0.7mm);
}
\end{tikzpicture}
\caption{$7$ points with $9$ double-normal pairs}\label{7points}
\end{figure}
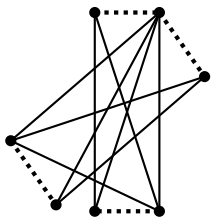

Note that, for even values of $n$, Theorem~\ref{thm:2d} can also be deduced from a result of Gr\"unbaum \cite{Gr} (see \cite{NS} for a proof), using Lemma~\ref{five}(\ref{1}) below.
For odd $n$, the same argument gives only the weaker bound $N(V) \leq 3\bigl\lfloor \frac{\card{V}}{2} \bigr\rfloor + 1$.

In \cite{PS}, we showed that the bounds in spaces of dimension $3$ and higher are quadratic, in particular, \[\lim_{n\to\infty} \frac{N_3(n)}{n^2}=\lim_{n\to\infty} \frac{N'_3(n)}{n^2}=\frac{1}{4}.\]
However, if we restrict a finite subset $V$ of $n$ points in $\bR^3$ to be on the $2$-sphere, then $N(V)$ grows at most linearly in $\card{V}$.

First, note that for any $n\geq 4$ except $n=5$, there exist $n$ points on a $2$-sphere with $2n-2$ diameter pairs.
This matches the maximum number of diameter pairs in $\bR^3$ \cite[Lemma 7(e)]{sw-lenz}.
Since diameter pairs are strict double-normal pairs, it follows that there exist $n$ points on the $2$-sphere with at least $2n-2$ strict double-normal pairs.
This cannot be improved.
\begin{theorem}\label{thm:0}
Given a finite set $V$ of at least $4$ points on a $2$-sphere, the number of strict double-normal pairs in $V$ \textup{(}as a subset of $\bR^3$\textup{)} satisfies \[N(V)\leq 2\card{V}-2.\]
This bound is sharp for each $\card{V}\geq 4$.
\end{theorem}
What happens if we wish to bound the number of not necessarily strict double-normals?
The vertex set of the cube in $\bR^3$ shows that $N_3(n)=\binom{n}{2}$ for $n\leq 8$.
However, our next theorem shows that even in this case there is a linear upper bound on the number of double-normals.
\begin{theorem}\label{thm:sphere}
Given a finite set $V$ of at least~$8$ points on a $2$-sphere, the number of double-normal pairs in $V$ \textup{(}as a subset of $\bR^3$\textup{)} satisfies
\[N(V)\leq\frac{17}{4}\card{V}-6\text{.}\]
If equality holds, then $V$ is symmetric around the centre of the sphere, and the faces of the convex hull of $V$ are rectangles and acute triangles, with each vertex belonging to exactly $3$ rectangular faces.

Conversely, for any finite subset $V$ of the $2$-sphere symmetric around the centre of the $2$-sphere, such that the faces of its convex hull are rectangles and triangles, with each vertex belonging to exactly $3$ rectangles, we have $N(V)=\frac{17}{4}\card{V}-6$.
\end{theorem}
The vertex sets of the cube and the vertex set of the small rhombicuboctahedron are two examples where $N(V)=\frac{17}{4}\card{V}-6$ (with $\card{V}=8$ and $\card{V}=24$, respectively).
We asymptotically match this upper bound up to an error of $O(\sqrt{\card{V}})$.
\begin{theorem}\label{thm:sphereconstruction}
For each $n$, there exists a set of $n$ points on the $2$-sphere with at least $\frac{17}{4}n-O(\sqrt{n})$ double-normal pairs as $n\to\infty$.

For infinitely many values of $n$, there exist sets of $n$ points on the $2$-sphere with exactly $\frac{17n}{4}-6$ double-normal pairs.
\end{theorem}

The paper is structured as follows.
In the next section, we present the proof of Theorem~\ref{thm:2d}.
In Section~\ref{sect:gabriel}, we introduce certain variants of Gabriel graphs for points on the $2$-sphere and study them using Euler's formula and the Delaunay tiling of these points.
We apply these results to prove Theorem~\ref{thm:0} in Section~\ref{sect:3}, Theorem~\ref{thm:sphere} in Section~\ref{sect:4}, and Theorem~\ref{thm:sphereconstruction} in Section~\ref{sect:5}.

\section{Proof of Theorem~\ref{thm:2d}}
This proof is based on Perles' proof that in a geometric graph where any two non-adjacent edges cross, the number of edges is at most the number of vertices \cite[Theorem~9]{Pach}.

Let $V$ be a set of $n$ points in the plane.
We draw its double-normal graph by joining each double-normal pair with a straight-line segment.
In the sequel, if it leads to no confusion, these segments will also be referred to as ``edges''. (Note that the resulting drawing is not necessarily a ``geometric graph'' in the sense the term is usually used in the literature \cite[Chapter~10]{CDGHandbook}, because it may have a vertex which lies in the relative interior of an edge.)
The following properties of this drawing are easily verified:
\begin{claim}\label{five}
\mbox{}
\begin{enumerate}[\textup{(}i\textup{)}]
\item\label{2} Two edges cannot lie on the same line.
In particular, two edges can intersect in at most one point.
\item\label{1} If $\vx\in V$ lies in the relative interior of an edge $\vy\vz\in E$, then $\vx$ is joined to at most one vertex $\vv\in V$, and then $\vx\vv$ must be perpendicular to $\vy\vz$.
\item\label{3} Any two disjoint edges are opposite edges of some rectangle.
\item\label{4a} No vertex lies in the convex hull of its neighbours.
\item\label{4b} All non-isolated vertices are vertices of the convex hull of $V$.
\end{enumerate}
\end{claim}
We define the edge $\vx\vy$ to be a \emph{rightmost edge} at the vertex $\vx$ if the half-plane bounded by the line $\vx\vy$ which lies on the right-hand side of the vector $\overrightarrow{\vx\vy}$ contains no point of $S$ in its interior.
Colour the (unique) rightmost edge of each non-isolated vertex \emph{red}.
By Lemma~\ref{five}(\ref{4a}), each such vertex has a rightmost edge.
This gives at most $n$ red edges.
Colour all the remaining edges \emph{blue}.
We next show
\begin{claim}\label{5}
The blue edges form a matching.
\end{claim}
\begin{proof}
Suppose to the contrary that two blue edges have a common endpoint $\vb$.
We label the other endpoints $\va$ and $\vc$ so that $\va$ lies on the left-hand side of the vector $\overrightarrow{\vb\vc}$ (Fig.~\ref{fig3}).
The rightmost edge at $\va$ does not intersect $\vb\vc$, so forms a rectangle together with $\vb\vc$, by Lemma~\ref{five}(\ref{3}).
\begin{figure}
\centering
\begin{tikzpicture}[scale=1.1, thick, rotate=-15]
\draw[dashed, red, ultra thick] (0,1) -- (2,1) node[black, right] {$\va$};
\draw[dashed, red, ultra thick] (0,0) node[black, below] {$\vb$} -- (2,-0.4);
\draw[blue] (0,0) -- (2,0) node[black, right] {$\vc$};
\draw[blue] (0,0) -- (2,1);
\end{tikzpicture}
\caption{Blue edges must be disjoint. (In this and subsequent figures, red edges are drawn dashed and blue lines solid.)}\label{fig3}
\end{figure}
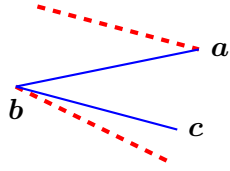
The rightmost edge at $\vb$ will also be disjoint from the rightmost edge at $\va$.
By Lemma~\ref{five}(\ref{3}), they also form a rectangle.
However, then the rightmost edge at $\vb$ coincides with $\vb\vc$, contradicting Lemma~\ref{five}(\ref{2}).
\end{proof}
Denoting the set of red edges by $R$ and the set of blue edges by $B$, we now already have
\begin{equation}
\card{E} = \card{R} + \card{B} \leq n+ \frac{n}{2}, \label{count}
\end{equation}
which is the required inequality when $n$ is even.

In the case where $n$ is odd, we only obtain $\card{E}\leq 3\lfloor n/2\rfloor +1$.
To finish the odd case, we have to analyze the graph $G$ further.
Along the way, we characterize equality in \eqref{count} for even $n$.
We say that two edges \emph{cross} if they share interior points.
\begin{claim}\label{6}
Any two blue edges cross.
\end{claim}
\begin{proof}
Suppose to the contrary that the blue edges $\va\vb$ and $\vc\vd$ do not cross.
By Lemma~\ref{5}, they do not share an endpoint.
Then either the segments $\va\vb$ and $\vc\vd$ are disjoint, or one of the segments, say $\va\vb$, has an endpoint, say $\va$, in the interior of the other segment $\vc\vd$ (Fig.~\ref{fig4}).
\begin{figure}
\centering
\begin{tikzpicture}[thick, scale=1.1]
\draw[blue] (0,1) node[black,left] {$\va$} -- (2,1) node[black,right] {$\vb$};
\draw[blue] (0,0) node[black,left] {$\vc$} -- (2,0) node[black,right] {$\vd$};
\begin{scope}[xshift=4cm]
\draw[blue] (0,1) node[black,left] {$\vc$} -- (2,1) node[black,right] {$\vd$};
\draw[blue] (1,1) node[black,above] {$\va$} -- (1,0) node[black,right] {$\vb$};
\end{scope}
\end{tikzpicture}
\caption{Two blue edges cannot be disjoint, nor can the endpoint of one lie in the interior of the other.}\label{fig4}
\end{figure}
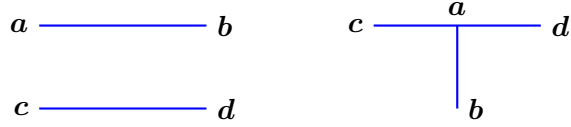
In the first case, the red edge at $\vb$ will be disjoint from $\vc\vd$, hence will form a rectangle with $\vc\vd$.
Since $\va\vb$ also forms a rectangle with $\vc\vd$, we obtain a contradiction.

In the second case, by Lemma~\ref{five}(\ref{1}), $\va$ has degree $1$, so $\va\vb$ is a red edge, which is a contradiction.
\end{proof}
\begin{claim}\label{7}
If a blue edge and a red edge do not have a common endpoint, then they cross.
\end{claim}
\begin{proof}
Otherwise, one of the three cases depicted in Fig.~\ref{fig5} will occur, where $\va\vb$ is blue and $\vc\vd$ is red, say.
In each case we arrive at a contradiction, as in the proof of Lemma~\ref{6}.
\begin{figure}
\centering
\begin{tikzpicture}[thick, scale=1.1]
\draw[blue] (0,1) node[black,left] {$\va$} -- (2,1) node[black,right] {$\vb$};
\draw[dashed,red, ultra thick] (0,0) node[black,left] {$\vc$} -- (2,0) node[black,right] {$\vd$};
\draw[dashed,red, ultra thick] (2,1) -- (1.4,1.3);
\begin{scope}[xshift=3.3cm]
\draw[dashed,red, ultra thick] (0,1) node[black,left] {$\vc$} -- (2,1) node[black,right] {$\vd$};
\draw[blue] (1,1) node[black,above] {$\va$} -- (1,0) node[black,right] {$\vb$};
\end{scope}
\begin{scope}[xshift=6.6cm]
\draw[blue] (0,1) node[black,left] {$\va$} -- (2,1) node[black,right] {$\vb$};
\draw[dashed,red, ultra thick] (1,1) node[black,above] {$\vc$} -- (1,0) node[black,right] {$\vd$};
\draw[dashed,red, ultra thick] (2,1) -- (1.4,1.3);
\end{scope}
\end{tikzpicture}
\caption{A blue edge and a red edge cannot be disjoint, nor can the endpoint of one lie in the interior of the other.}\label{fig5}
\end{figure}
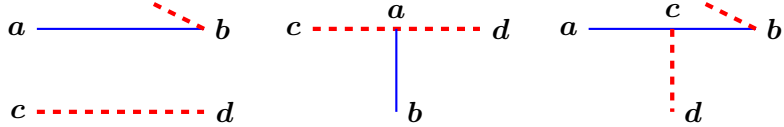
(In the last case, the red edge at $\vb$ would have to form a rectangle with $\vc\vd$, by Lemma~\ref{five}(\ref{3}), which is impossible.)
\end{proof}
We now characterize the case of equality when $n$ is even.
Assume $n$~is even and $\card{E}=n+n/2$.
To be consistent with \eqref{count}, there must be exactly $n$~red edges and $n/2$ blue edges.
In particular, no red edge is a rightmost edge of both of its endpoints, and no vertex is isolated.

Since the $n/2$ blue edges are pairwise crossing (Lemma~\ref{6}), the vertices have a natural cyclic order $\vp_1,\vp_2,\dots,\vp_n$ such that the blue edges are $\vp_i\vp_{i+n/2}$ ($i=1,\dots,n/2$); see Fig.~\ref{fig6}.
\begin{figure}
\centering
\begin{tikzpicture}[scale=0.67, thick]
\foreach \point/\theta in {p1/0,p2/45,p3/90,p4/135,p5/180,p6/225,p7/270,p8/315} {
    \coordinate (\point) at (\theta+22.5:2cm);
}
\draw[dashed,red, ultra thick] (p1)--(p6);
\draw[dashed,red, ultra thick] (p2)--(p5);
\draw[blue] (p1)--(p5);
\draw[blue] (p2)--(p6);
\draw[blue] (p3)--(p7);
\draw[blue] (p4)--(p8);
\draw[fill] (p1)  node[right] {$\vp_1$} circle (1mm);
\draw[fill] (p2)  node[right] {$\vp_2$} circle (1mm);
\draw[fill] (p3)  node[left] {$\vp_3$} circle (1mm);
\draw[fill] (p4)  node[left] {$\vp_4$} circle (1mm);
\draw[fill] (p5)  node[left] {$\vp_{1+4}$} circle (1mm);
\draw[fill] (p6)  node[left] {$\vp_{2+4}$} circle (1mm);
\draw[fill] (p7)  node[right] {$\vp_{3+4}$} circle (1mm);
\draw[fill] (p8)  node[right] {$\vp_{4+4}$} circle (1mm);
\end{tikzpicture}
\caption{Equality in the even case ($n=8$)}\label{fig6}
\end{figure}
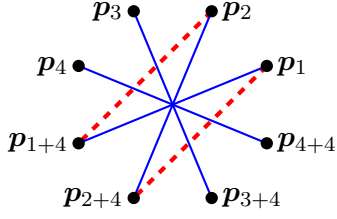

Let $i\in\{1,\dots,n\}$.
Since the red edge at $\vp_i$ is not disjoint from the blue edge $\vp_{i-1}\vp_{i-1+n/2}$ (Lemma~\ref{7}), it has to be the edge $\vp_i\vp_{i-1+n/2}$ (with subscripts taken modulo $n$).
This determines all the red edges.

The red edges $\vp_i\vp_{i-1+n/2}$ and $\vp_{i-1}\vp_{i+n/2}$ are disjoint, so by Lemma~\ref{five}(\ref{3}), they form a rectangle with diagonals the blue edges $\vp_i\vp_{i+n/2}$ and $\vp_{i-1}\vp_{i-1+n/2}$.
It follows that the blue edges all have the same midpoint and equal length.
Therefore, the points $\vp_1,\dots,\vp_n$ lie on a circle and are symmetric with respect to the centre of this circle.

Conversely, it is easy to see that any set of $n$ points on a circle, symmetric with respect to the centre of the circle, has $n+n/2$ double-normal pairs.

\smallskip
Suppose next that $n$ is odd and that
\[\card{E}=3\lfloor n/2\rfloor +1 = n + \frac{n-1}{2}.\]
We aim for a contradiction, which will finish the proof of Theorem~\ref{thm:2d}.

To be consistent with \eqref{count}, there must be exactly $n$ red edges and $(n-1)/2$ blue edges.
Thus, no red edge is the rightmost edge of both its endpoints, and by Lemma~\ref{6}, the blue edges form a pairwise crossing matching.

By Lemma~\ref{five}\eqref{4b}, there is a natural clockwise ordering $\vp_1,\dots,\vp_n$ of the points, which we choose in such a way that the blue edges are $\vp_i\vp_{i+(n-1)/2}$ ($i=1,\dots,(n-1)/2$), and with $\vp_n$ not incident to any blue edge (Fig.~\ref{fig7}).
\begin{figure}
\centering
\begin{tikzpicture}[scale=0.85, thick]
\foreach \point/\theta in {p1/0,p2/45,p3/90,p4/135,p5/180,p6/225,p7/270,p8/315} {
    \coordinate (\point) at (\theta+22.5:2cm);
}
\coordinate (p9) at (2.2,0);
\draw[dashed,red, ultra thick] (p1)--(p6);
\draw[dashed,red, ultra thick] (p2)--(p5);
\draw[dashed,red, ultra thick] (p3)--(p8);
\draw[dashed,red, ultra thick] (p4)--(p7);
\draw[dashed,red, ultra thick] (p4)--(p1);
\draw[dashed,red, ultra thick] (p6)--(p3);
\draw[dashed,red, ultra thick] (p7)--(p2);
\draw[blue] (p1)--(p5);
\draw[blue] (p2)--(p6);
\draw[blue] (p3)--(p7);
\draw[blue] (p4)--(p8);
\draw[fill] (p1)  node[right] {$\vp_1$} circle (1mm);
\draw[fill] (p2)  node[right] {$\vp_2$} circle (1mm);
\draw[fill] (p3)  node[left] {$\vp_3$} circle (1mm);
\draw[fill] (p4)  node[left] {$\vp_4$} circle (1mm);
\draw[fill] (p5)  node[left] {$\vp_{1+4}$} circle (1mm);
\draw[fill] (p6)  node[left] {$\vp_{2+4}$} circle (1mm);
\draw[fill] (p7)  node[right] {$\vp_{3+4}$} circle (1mm);
\draw[fill] (p8)  node[right] {$\vp_{4+4}$} circle (1mm);
\draw[fill] (p9)  node[right] {$\vp_{9}$} circle (1mm);
\end{tikzpicture}
\caption{Further analysis of the odd case ($n=9$)}\label{fig7}
\end{figure}
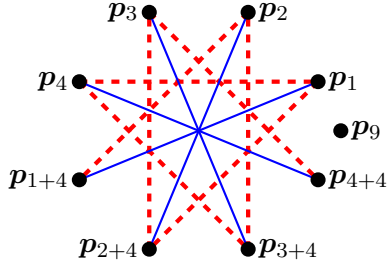

The rightmost edge of $\vp_1$ has to be $\vp_1\vp_{(n-1)/2}$, otherwise it would be disjoint from the blue edge $\vp_{(n-1)/2}\vp_{n-1}$, contradicting Lemma~\ref{7}.
Similarly, for each $i=1,\dots,(n-1)/2$, the rightmost edge of $\vp_i$ is $\vp_i\vp_{i-1+(n-1)/2}$,
and for each $i=(n+3)/2,\dots,n-1$, the rightmost edge of $\vp_i$ is $\vp_i\vp_{i-(n+1)/2}$.

There are two points for which we cannot determine the rightmost edges in this way:
The rightmost edge of $\vp_{(n+1)/2}$ could be either $\vp_{(n+1)/2}\vp_{n-1}$ or $\vp_{(n+1)/2}\vp_{n}$, and the rightmost edge of $\vp_n$ could be either $\vp_n\vp_{(n-1)/2}$ or $\vp_n\vp_{(n+1)/2}$.

For each $i=1,\dots,(n-3)/2$, the red edges $\vp_i\vp_{i+(n+1)/2}$ and $\vp_{i+1}\vp_{i+(n-1)/2}$ are disjoint.
By Lemma~\ref{five}(\ref{3}), they form a rectangle with diagonals the blue edges $\vp_i\vp_{i+(n-1)/2}$ and $\vp_{i+1}\vp_{i+(n+1)/2}$.
Thus, the blue edges all have the same midpoint and equal length.
It follows that $\vp_1\vp_{(n-1)/2}\vp_{(n+1)/2}\vp_{n-1}$ also forms a rectangle.
Since the rightmost edge of $\vp_{(n+1)/2}$ is disjoint from $\vp_1\vp_{(n-1)/2}$, hence is parallel to $\vp_1\vp_{(n-1)/2}$ (again Lemma~\ref{five}(\ref{3})), it must be $\vp_{(n+1)/2}\vp_{n-1}$.
However, it now follows that the rightmost edge at $\vp_n$ can neither be $\vp_n\vp_{(n+1)/2}$, since it would then have to be parallel to $\vp_1\vp_{(n-1)/2}$, nor can it be $\vp_n\vp_{(n-1)/2}$, since it would then have to be parallel to $\vp_{(n+1)/2}\vp_{n-1}$.
This contradiction shows that the inequality in \eqref{count} must be strict, and it follows that $\card{E}\leq3\lfloor n/2\rfloor$ when $n$ is odd.
This completes the proof of Theorem~\ref{thm:2d}.

\section{Gabriel graphs and Delaunay tilings on the sphere}\label{sect:gabriel}
In this section, we introduce strict and weak Gabriel graphs of sets of points on a $2$-sphere.
Strict Gabriel graphs can be considered to be the spherical analogue of
the standard Gabriel graphs \cite{Gabriel-Sokal, Matula-Sokal}.
They will be used to prove Theorem~\ref{thm:0} on strict double-normals.
Weak Gabriel graphs will be used to prove Theorems~\ref{thm:sphere} and \ref{thm:sphereconstruction}.
In Theorem~\ref{lemma:sphere} below, we determine the maximum number of edges of weak Gabriel graphs, using a notion of Delaunay tilings for points on a $2$-sphere.

Denote the unit sphere in $\bR^3$ by $\bS^2$ and its centre by $\vo$.
We call two points $\vx,\vy\in\bS^2$ \define{antipodal} if $\vy=-\vx$.

Let $V$ be a finite subset of $\bS^2$.
In the \emph{weak Gabriel graph} of $V$, two points $\va$ and $\vb$ are joined by an edge if and only if they are not antipodal and if no point of $V$ is contained in the interior of the minor spherical cap with diameter $\va\vb$.
The \emph{strict Gabriel graph} of $V$ is defined similarly, except that we furthermore require that no point of $V$ is on the boundary of the minor spherical cap with diameter $\va\vb$.
Note that we do not joint antipodal pairs in either graph.

We draw the strict and weak Gabriel graph of $V$ on $\bS^2$ by drawing the minor great-circular arc from $\va$ to $\vb$ for each $\va\vb\in E$.
As in the previous section, if there is no danger of confusion, we make no notational or terminological distinction between a strict or weak Gabriel graph and its drawing.

\begin{lemma}\label{crossing}
Two crossing arcs in the drawing of a weak Gabriel graph on $\bS^2$ have the same length and the same midpoint, which is also the point where they cross.

There are no crossings in the drawing of a strict Gabriel graph on $\bS^2$.
\end{lemma}
\begin{proof}
Let $\va\vb$ and $\vc\vd$ be two arcs of the weak Gabriel graph intersecting in $\vs$, say.
Let the midpoint of the arc $\va\vb$ be $\vp$ and the midpoint of $\vc\vd$ be $\vq$.
Without loss of generality, $\vp$ is on the arc $\vs\vb$ and $\vq$ is on the arc $\vs\vd$.
Since $\vd$ is not in the interior of the circle with diameter $\va\vb$, we have the inequality $\vp\vd\geq\vp\vb$ between the spherical lengths of the arcs.
This implies $\myangle\va\vb\vd\geq\myangle\vp\vd\vb$.
Similarly, since $\vp\vd\geq\vp\va$, we have $\myangle\vd\va\vb\geq\myangle\vp\vd\va$,
and since $\vp\vc\geq\vp\va$ and $\vp\vc\geq\vp\vb$, we also obtain $\myangle\vb\va\vc\geq\myangle\vp\vc\va$ and $\myangle\va\vb\vc\geq\myangle\vp\vc\vb$.
It follows that in the spherical quadrilateral $\va\vb\vc\vd$, $\myangle\va+\myangle\vb\geq\myangle\vc+\myangle\vd$.
Using $\vq$ instead of $\vp$, we similarly find that $\myangle\vc+\myangle\vd\geq\myangle\va+\myangle\vb$.
Therefore, all inequalities become equalities.
It follows that $\va\vb\vc\vd$ is inscribed in a circle with centre $\vp=\vq=\vs$ and diameters $\va\vb$ and $\vc\vd$.
This implies the first statement of the lemma, and also that $\va\vb$ and $\vc\vd$ cannot belong the strict Gabriel graph, which gives the second statement.
\end{proof}

We next introduce the Delaunay tiling of a finite set of points on $\bS^2$, which is needed in the description of weak Gabriel graphs with a maximum number of edges.
We first define a \emph{spherical polygon} to be the intersection of finitely many non-opposite closed hemispheres of $\bS^2$, such that the intersection has non-empty interior and does not contain antipodal pairs of points.
The boundary of a spherical polygon consists of $k$ vertices and $k$ minor great-circular arcs, for some $k\geq 3$.
Given a finite subset $V$ of $\bS^2$, form its convex hull $P:=\conv V$ in $\bR^3$.
A point $\vp\in P$ is an \emph{outside point} of $P$ if $P$ is disjoint from the open ray $\setbuilder{\lambda \vp}{\lambda>1}$.
All vertices of $P$ are outside points of $P$, and all outside points of $P$ are boundary points of $P$.
An edge or face of $P$ is called \emph{outside} if all of its points are outside points.
The \emph{Delaunay tiling} of $P$ is defined to consist of the vertices $V$ of $P$ and the central projections of the outside edges and faces of $P$ from $\vo$ to $\bS^2$.
The \emph{edges} of the Delaunay tiling of $P$ are the minor great-circular arcs that are the projections of the outside edges of $P$, and the \emph{faces} of the Delaunay tiling are the projections of the outside faces of $P$.
Thus, the Delaunay tiling is a tiling
\begin{enumerate}[\quad(a)]
\item of the whole $\bS^2$ if $\vo$ is in the interior of $P$,
\item of a hemisphere of $\bS^2$ if $\vo$ is in the relative interior of a face of $P$,
\item of the intersection of two hemispheres of $\bS^2$ if $\vo$ is in the relative interior of an edge of $P$,
\item and finally, of the smallest spherical polygon that contains $V$ if $\vo\notin P$.
\end{enumerate}
\begin{lemma}\label{delaunaycrossing}
No edge of the weak Gabriel graph of $V$ crosses an edge of the Delaunay tiling of $V$.
\end{lemma}
\begin{proof}
Consider an edge $\va\vb$ of the weak Gabriel graph $G$ of $V$.
Note that the plane that passes through the boundary of the minor spherical cap with diameter $\va\vb$, supports $P$.
It follows that each edge of $G$ is contained in some face of the Delaunay tiling $D$ of $V$.
\end{proof}
The main result of this section is the following upper bound for the number of edges of a weak Gabriel graph, together with a characterization of equality.
\begin{theorem}\label{lemma:sphere}
The weak Gabriel graph $G$ of a finite set $V$ of at least $2$~points on $\bS^2$ has at most $\frac{15}{4}\card{V}-6$ edges.
If equality occurs, then the interior of the convex hull of $V$ contains the origin $\vo$, and each face of the Delaunay tiling of $V$ is either an acute spherical triangle or an equiangular spherical quadrilateral, each vertex is incident to exactly $3$ spherical quadrilaterals, and the edges of $G$ are the edges of the Delaunay tiling together with the diagonals of the spherical quadrilaterals.

Conversely, if a finite subset $V$ of $\bS^2$ is given such that $\vo$ is in the interior of its convex hull, and such that the faces of its convex hull are rectangles and triangles, with $3$ rectangles at each vertex, then the weak Gabriel graph of $V$ has exactly $\frac{15}{4}\card{V}-6$ vertices.
\end{theorem}
\begin{proof} 
Define a relation $\sim$ on the set $E$ of edges of $G$ by setting $\ve_1\sim\ve_2$ if $\ve_1=\ve_2$ or $\ve_1$ crosses $\ve_2$.
By Lemma~\ref{crossing}, $\sim$ is an equivalence relation on $E$, where each equivalence class is composed of edges drawn as congruent arcs with a common midpoint.
Note that although crossings of arcs may occur, by the definition of a weak Gabriel graph, no point in $V$ can be in the relative interior of an arc.

Without loss of generality, $\card{E}\geq 2$.
Consider an equivalence class of at least two pairwise crossing arcs.
There is a unique spherical polygon such that its vertices are exactly the endpoints of the crossing arcs, with each arc a diagonal.
We call this spherical polygon a \emph{crossing polygon}.
\begin{claim}\label{star2}
If two crossing polygons intersect, then they intersect in either a single vertex or in a common edge.
\end{claim}
\begin{proof}
Denote the two intersecting crossing polygons by $P_1$ and $P_2$.
Let $C_i$ be the circumcircle of $P_i$ ($i=1,2$).
The claim is obvious if $C_1$ and $C_2$ touch in a single point.
Thus we may assume that $C_1$ and $C_2$ intersect in two points $\vp$ and $\vq$, say.
By the definition of the weak Gabriel graph $G$, no point of $V$ is in the interior of either $C_1$ or $C_2$.
Thus, the vertices of $P_i$ all lie on the major arc of $C_i$ from $\vp$ to $\vq$ ($i=1,2$).
If neither $\vp$ nor $\vq$ is a common vertex of $P_1$ and $P_2$, then $P_1$ and $P_2$ are disjoint, a contradiction.
Therefore, $P_1$ and $P_2$ either have one vertex ($\vp$ or $\vq$) in common and no other point, or have both vertices $\vp$ and $\vq$ in common, and then they have an edge in common.
\end{proof}
We now modify the weak Gabriel graph $G$ to form a new graph $G'=(V,E')$ on the same vertex set, drawn on $\bS^2$ as follows.
For each equivalence class of at least two pairwise crossing arcs, remove the crossing arcs, and add the edges of the associated crossing polygon if they are not already in $G$ (Figure~\ref{conversion}).
\begin{figure}
\centering
\begin{tikzpicture}[thick, scale=0.6]
\draw[blue] (0,0) circle (2cm);
\foreach \theta in {0, 50, 90, 150} {
    \draw[fill] (\theta:2cm) circle (1mm);
    \draw[fill] (\theta+180:2cm) circle (1mm);
}
\foreach \theta in {0, 50, 90, 150} {%
\draw (\theta:2cm) -- (\theta+180:2cm);}%
\draw[dashed] (0:2cm) \foreach \theta in {50, 90, 150} {%
-- (\theta:2cm)}%
\foreach \theta in {0, 50, 90, 150} {
-- (\theta+180:2cm)}
-- cycle;
\end{tikzpicture}
\qquad\raisebox{1.3cm}{\scalebox{2}{$\rightsquigarrow$}}
\qquad
\begin{tikzpicture}[thick, scale=0.6]
\draw[blue] (0,0) circle (2cm);
\foreach \theta in {0, 50, 90, 150} {
    \draw[fill] (\theta:2cm) circle (1mm);
    \draw[fill] (\theta+180:2cm) circle (1mm);
}
\foreach \theta in {0, 50, 90, 150} {%
\draw[dash pattern=on 3.3pt off 3pt] (\theta:2cm) -- (\theta+180:2cm);}%
\draw (0:2cm) \foreach \theta in {50, 90, 150} {%
-- (\theta:2cm)}%
\foreach \theta in {0, 50, 90, 150} {
-- (\theta+180:2cm)}
-- cycle;
\end{tikzpicture}
\caption{Creating $G'$ from the weak Gabriel graph $G$}\label{conversion}
\end{figure}
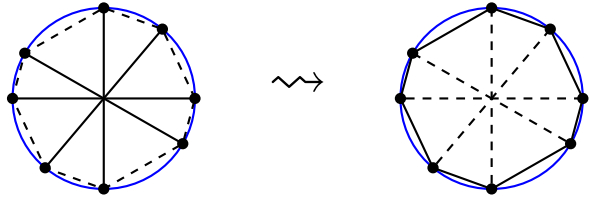
By Lemma~\ref{star2}, no edge of a crossing polygon can also be an edge of $G$ that crosses some other edge of $G$, and therefore, $G'$ is unambiguously defined.
Also, since $\card{E}\geq 2$, it follows that $\card{E'}\geq 2$ (either no new edges are added, or there is a crossing polygon with at least $4$ edges and then $\card{E'}\geq 4$).
\begin{claim}\label{11}
No edge of $G'$ contains a vertex in its relative interior.
\end{claim}
\begin{proof}
As mentioned before, no edge of $G$ contains a vertex in its relative interior.
Moreover, if a newly added edge $e'$ passed through some $\vp\in V$, then the spherical cap circumscribing the crossing polygon to which $e'$ belongs would contain $\vp$, which would contradict the defining property of the weak Gabriel graph $G$.
\end{proof}
\begin{claim}\label{12}
$G'$ is drawn without crossings.
\end{claim}
\begin{proof}
By construction we have eliminated crossings between edges of $G$.
Suppose that a newly added edge $e'\in E'\setminus E$ crosses an edge $e\in E\cap E'$ of $G'$ that was already in $G$.
Since no vertex lies inside the crossing polygon of which $e'$ is an edge, $e$ has to cross the whole crossing polygon, and in particular, one of its diagonals $f$, say (Figure~\ref{fig:crossing}).
\begin{figure}
\centering
\begin{tikzpicture}[thick, scale=0.7]
\draw[blue] (0,0) circle (2cm);
\foreach \theta in {0, 60} {
    \draw[fill] (\theta:2cm) circle (1mm);
    \draw[fill] (\theta+180:2cm) circle (1mm);
}
\foreach \theta in {0, 60} {%
\draw[dashed] (\theta:2cm) -- (\theta+180:2cm);}%
\draw (0:2cm) -- node [pos=0.55, left] {$e'$} (60:2cm);
\draw (0+180:2cm) -- (60+180:2cm);
\draw (10:0.7cm)  -- (20:2.8cm) node[pos=0.75, above] {$e$} node[very near start, below=1mm] {$f$};
\draw[fill] (20:2.8cm) circle (1mm);
\end{tikzpicture}
\caption{$G'$ has no crossings}\label{fig:crossing}
\end{figure}
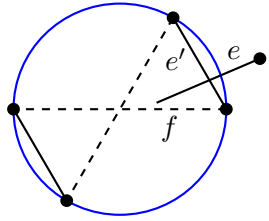
By Lemma~\ref{crossing}, $e$ and $f$ cross at their common centre and they have the same length.
It follows that $e$ is also a diagonal of the crossing polygon, which contradicts that $e\in E'$.
Finally, two newly added edges $e',f'\in E'\setminus E$ cannot cross by Lemma~\ref{star2}.
\end{proof}
Lemmas~\ref{11} and \ref{12} together show that $G'$ is embedded in $\bS^2$.
As usual, we define the \emph{faces} of $G'$ to be the connected components of the complement in $\bS^2$ of the drawing of $G'$.
We define the number of edges \emph{bounding} a face $F$ as the number of arcs belonging to the boundary of $F$, with the convention that an arc is counted twice if $F$ is on both sides of it.
Let $f_i$ denote the number of faces of $G'$ bounded by $i$ edges.
Since $\card{E'}\geq2$, $f_0=f_1=f_2=0$.
(Note that antipodal pairs are not joined in $G'$.)
Then the following well-known inequality holds:
\begin{equation}\label{three'}
\card{E'}\leq 3\card{V}-6-f_4-2f_5-3f_6-\cdots.
\end{equation}
Indeed, counting incident vertex--edge pairs in two ways gives
\begin{equation}\label{one}
2\card{E'}= 3f_3+4f_4+5f_5+\cdots,
\end{equation}
and if we denote the number of connected components of $G'$ by $c'$, then
by Euler's formula,
\begin{equation}\label{three}
\card{V}-\card{E'}+f_3+f_4+\cdots = 1+c'\geq 2.
\end{equation}
Now add $3\times$\eqref{three} to~\eqref{one} to obtain \eqref{three'}.

Let $g_i$ denote the number of crossing polygons with $i$ edges.
Then $g_i=0$ unless $i$ is even and $i\geq 4$.
Also,
\begin{equation}\label{gf}
g_i\leq f_i\quad\text{for all $i$.}
\end{equation}
Each angle of a crossing polygon is obtuse.
Therefore, each vertex is incident to at most three crossing polygons.
Counting incident vertex--crossing polygon pairs in two ways, we obtain:
\begin{equation*}\label{twoa}
4g_4+6g_6+\dots\leq 3\card{V},
\end{equation*}
hence
\begin{equation}\label{twob}
g_4\leq \frac{3}{4}\card{V}.
\end{equation}
For each crossing polygon with $i$ edges, at most $i/2$ edges were removed from~$G$.
Therefore, the number of original edges in $G$ is at most
\begin{align}
\card{E} &\leq \card{E'}+2g_4+3g_6+\cdots\notag\\
&\overset{\eqref{three'}}{\leq} 3\card{V}-6 -f_4-2f_5-3f_6-\cdots + 2g_4 +3g_6 +4g_8 + \cdots\notag\\
&\overset{\eqref{gf}}{\leq} 3\card{V}-6 + g_4\notag\\
&\overset{\eqref{twob}}{\leq} 3\card{V}-6 + \frac{3}{4}\card{V},\notag
\end{align}
which proves the first part of the theorem.
Equality implies that $g_6=g_8=g_{10}=\dots=0$, $f_5=f_6=f_7=\dots=0$, $f_4=g_4=3\card{V}/4$ and $c'=1$.
That is, the only crossing polygons in $G'$ are spherical quadrilaterals, each quadrilateral face of $G'$ is a crossing polygon and is therefore equiangular, the edges of the crossing polygons were already in $G$, the only faces of $G'$ are spherical triangles and spherical quadrilaterals, each vertex is incident to exactly three spherical quadrilaterals, and $G'$ is connected.

It follows that the angles of the spherical triangles are all acute, and in particular, the spherical triangles cannot contain an open hemisphere.
It also follows that the angles of the spherical quadrilaterals must all be less than $\pi$, which means that no spherical quadrilateral contains an open hemisphere.
Therefore, $\vo$ is in the interior of $P:=\conv V$, and the central projections of the faces and edges of $P$ from $\vo$ onto $\bS^2$ form the Delaunay tiling of $V$.
Consider a spherical quadrilateral face $\va\vb\vc\vd$ of $G'$.
Since the edges $\va\vb, \vb\vc, \vc\vd, \vd\va\in E$, the circles with these edges as diameters do not contain any vertex in their interiors.
In particular, the circumcircle of $\va\vb\vc\vd$ does not pass through any point of $V$ other than $\va$, $\vb$, $\vc$, $\vd$.
Since $\va$, $\vb$, $\vc$, $\vd$ lie in a plane, it follows that $\conv\set{\va,\vb,\vc,\vd}$ is a rectangular face of $P$.
Therefore, $\va\vb\vc\vd$ is a face of the Delaunay tiling.

Similarly, given a triangular face $\va\vb\vc$ of $G'$, the circles with diameters $\va\vb$, $\vb\vc$, $\vc\va$ contain the circumcircle of $\va\vb\vc$, and it follows that $\conv\set{\va,\vb,\vc}$ is a triangular face of $P$.

It follows that $G'$ is the graph of the Delaunay tiling of $V$.
If we add the diagonals of each quadrilateral face, we obtain the original weak Gabriel graph $G$.

\medskip
Suppose next $V\subset\bS^2$ is given such that $\vo$ is in the interior of $P:=\conv V$, and such that the faces of $P$ are rectangles and triangles, with each vertex belonging to three rectangles.
Then the triangles are necessarily acute.
Also, the faces of the Delaunay tiling are equiangular spherical quadrilaterals and acute spherical triangles.
We have to show that the graph of the Delaunay tiling and the diagonals of the equiangular quadrilaterals together form the weak Gabriel graph of $V$.

By Lemma~\ref{delaunaycrossing}, for each edge $\va\vb$ of the weak Gabriel graph $G$, the segment $\va\vb$ is on the boundary of $P$, hence is either an edge of $P$ or a diagonal of one of the rectangular faces.
It follows that the arc $\va\vb$ is either an edge of the Delaunay triangulation $D$ or a diagonal of a quadrilateral face of $D$.

Conversely, we have to show that the edges of $D$ and the diagonals of the quadrilateral faces of $D$ are also edges of $G$.
Consider first the diagonal $\va\vc$ of a quadrilateral face $\va\vb\vc\vd$ of $D$.
The circle $C_1$ with diameter $\va\vc$ circumscribes $\va\vb\vc\vd$.
Since the plane through $C_1$ supports $P$ in the rectangle $\va\vb\vc\vd$, it follows that no other vertex of $V$ lie inside or on $C_1$, hence $\va\vc$ is an edge of $G$.

Next, consider an edge $\va\vb$ of $D$.
We have to show that no other point of $V$ lies inside the circle $C_2$ with diameter $\va\vb$.
Let $F$ be one of the two faces of $D$ bounded by $\va\vb$.
Since the vertices of $F$ are not in the interior of $C_2$, the circumcircle of $F$ (which contains no points of $V$ other than the vertices of $F$), contains the one semicircle of $C_2$ bounded by $\va\vb$.
Similarly, the circumcircle of the other face bounded by $\va\vb$ contains the other semicircle of $C_2$.
It follows that $C_2$ does not have any point of $V$ in its interior.

We have shown that the edges of the weak Gabriel graph of $V$ are exactly the edges of the Delaunay triangulation together with the diagonals of the quadrilateral faces.
Similar to the calculation above, it now easily follows that the weak Gabriel graph of $V$ has exactly $\frac{15}{4}\card{V}-6$ edges.
\end{proof}

\section{Proof of Theorem~\ref{thm:0}}\label{sect:3}
We next use strict Gabriel graphs to prove Theorem~\ref{thm:0}.
The statement is trivial when $\card{V}=4$, so we assume that $\card{V}>4$ and that the theorem holds for sets of smaller size.

Suppose that $V$ contains two antipodal points $\vx$ and $\vy$ (that is, $\vy=-\vx$).
Then $\vx\vy$ is a strict double-normal pair.
We claim that $\vx$ and $\vy$ have no other neighbours in the strict double-normal graph of $V$.
Indeed, if $\vx\vz$ is another double-normal pair, say, then the plane through $\vz$ perpendicular to $\vx\vz$ contains $\vy$, so $\vx\vz$ is not a strict double-normal pair.
It follows that \[N(V)=N(V\setminus\set{\vx,\vy}) + 1 \leq 2(\card{V}-2)-2+1 < 2\card{V}-2.\]

Therefore, we may assume without loss of generality that $V$ does not contain antipodal pairs of points.
For any $\vx\in\bS^2$, write $\vx'$ for the antipodal point $-\vx$ of $\vx$ on $\bS^2$, and let $V':=\setbuilder{\vv'}{\vv\in V}$.
By assumption, $V\cap V'=\emptyset$.
Define a graph $G$ on $V\cup V'$ with edge set
\[ E:= \setbuilder{\vx\vy'}{\vx,\vy\in V, \vx\vy\text{ is a strict double-normal pair in }V}.\]
Draw the edges of $G$ as minor great-circular arcs of $\bS^2$.
\begin{claim}\label{claim0}
$G$ is contained in the strict Gabriel graph of $V\cup V'$.
\end{claim}
\begin{proof}
For any strict double-normal pair $\vx\vy$ of $V$, since $\vx$ and $\vy$ are not antipodal, the planes through $\vx$ and $\vy$ perpendicular to the chord $\vx\vy$, intersect $\bS^2$ in the circles with diameters $\vx\vy'$ and $\vx'\vy$.
Because $\vx\vy$ is a strict double-normal pair of $V$, no point of $V$ or $V'$ lies on or in the interior of the circular caps cut off by these planes.
It follows that $\vx\vy'$ and $\vx'\vy$ are edges of the strict Gabriel graph of~$V\cup V'$.
\end{proof}
By Lemmas~\ref{crossing} and \ref{claim0}, $G$ is planar.
By construction, $G$ is bipartite with classes $V$ and $V'$.
By a well-known consequence of Euler's formula, we obtain $\card{E}\leq 2(2\card{V})-4$.
Since the graph $G$ has two edges $\vx\vy'$ and $\vx'\vy$ for each strict double-normal pair of $V$,
we obtain $2N(V)=\card{E}\leq 4\card{V}-4$,
and the first part of the theorem follows.

As mentioned before, for each $n\geq 4$, except $n=5$, there exists a set of $n$ points on the $2$-sphere with $2n-2$ diameters \cite[Lemma 7(e)]{sw-lenz}.
This shows that the inequality is sharp, except possibly for $n=5$.
However, it is not difficult to find $5$ points on the sphere with $8$ strict double-normal pairs.
Indeed, let $\vp_1,\vp_2,\vp_3$ be three equidistant points on some great circle $C_1$ of $\bS^2$.
Let $C_2$ be the great circle that passes through $\vp_3$ perpendicular to $C_1$.
Let $\vp_4$ and $\vp_5$ be points on $C_2$ close to $\vp_3$, with $\vp_3$ between $\vp_4$ and $\vp_5$.
Then $\set{\vp_1,\dots,\vp_5}$ has $8$ strict double-normal pairs (all pairs except $\vp_3\vp_4$ and $\vp_3\vp_5$).
This finishes the proof of Theorem~\ref{thm:0}.

\section{Proof of Theorem~\ref{thm:sphere}}\label{sect:4}
As in the proof of Theorem~\ref{thm:0}, write $\vx'$ for the antipodal point $-\vx$ of $\vx$, and let $V':=\setbuilder{\vv'}{\vv\in V}$.
Define a graph $G_1$ on $V\cup V'$ with edge set
\[ E_1:= \setbuilder{\vx\vy'}{\vx,\vy\in V, \vx\neq \vy', \vx\vy\text{ is a double-normal pair in }V}.\]
Draw the edges of $G_1$ as minor great-circular arcs of $\bS^2$.
Let $G_2=(V\cap V', E_2)$ be the induced subgraph of $G_1$ on $V\cap V'$.
\begin{claim}\label{15}
$G_1$ is contained in the weak Gabriel graph of $V\cup V'$, and $G_2$ is contained in the weak Gabriel graph of $V\cap V'$.
\end{claim}
\begin{proof}
The fact that $G_1$ is a subgraph of the weak Gabriel graph of $V\cup V'$ is shown in the same way as Lemma~\ref{claim0} in the proof of Theorem~\ref{thm:0}.

If $\vx\vy$ is an edge of $G_2$, then $\vx,\vy\in V\cap V'$, and $\vx\vy$ is a double-normal pair of $V$.
Therefore, both $\vx\vy$ and $\vx'\vy'$ are double-normal pairs of $V\cap V'$.
As before, $\vx\vy'$ and $\vx'\vy$ are edges of the weak Gabriel graph of~$V\cap V'$.
\end{proof}
\begin{claim}\label{16}
$2N(V) = \card{E_1}+\card{E_2}+\card{V\cap V'}$.
\end{claim}
\begin{proof}
Each double-normal pair $\vx\vy$ of $V$, where $\vx\neq\vy'$, is represented by two edges $\vx\vy'$ and $\vx'\vy$ of $G_1$.
If in addition $\vx',\vy'\in V$, then $\vx'\vy'$ is also a double-normal pair of $V$, but represented by the same two edges $\vx\vy'$ and $\vx'\vy$ of $G_1$.
However, then these two edges are in $G_2$.
If $\vx=\vy'$, then $\vx$ and $\vy$ are antipodal points and correspond to the two points $\vx,\vy\in V\cap V'$.
\end{proof}
By Lemmas~\ref{15} and \ref{16}, and Theorem~\ref{lemma:sphere}, we obtain the upper bound
\begin{align*}2N(V)&\leq\dfrac{15}{4}\card{V\cup V'}-6 + \dfrac{15}{4}\card{V\cap V'}-6 + \card{V\cap V'}\\
&= \frac{15}{2}\card{V}-12+\card{V\cap V'}\\
&\leq\frac{17}{2}\card{V}-12,
\end{align*}
hence $N(V)\leq \frac{17}{4}\card{V}-6$.
Equality implies that $\card{V}=\card{V\cap V'}$ and that equality holds in Theorem~\ref{lemma:sphere}.
Thus, $V=V'$, and the faces of $\conv V$ are rectangles and triangles, with exactly three rectangles at each vertex.
This concludes the proof of Theorem~\ref{thm:sphere}.

\section{Proof of Theorem~\ref{thm:sphereconstruction}}\label{sect:5}
We start with a construction.
\begin{lemma}\label{construction}
For any even $k\geq 4$ and any $m\geq 1$, there exists a set $V\subset\bS^2$ such that $\card{V}=2(2^{m}-1)k$ and $N(V)=\frac{17}{4}\card{V}-\frac{3}{2}k$.
\end{lemma}
\begin{proof}
Let $\vp$ denote the north pole on $\bS^2$, and let $C_0, C_1, \dots, C_{m-1}$ be circles in the northern hemisphere of $\bS^2$ equidistant from $\vp$ (that is, lines of latitude), with their radii chosen in such a way that we can inscribe a regular $2^ik$-gon in $C_i$ such that all $m$ polygons have the same spherical side length.
Since it is possible to do this in the plane, it is also possible on $\bS^2$ in a sufficiently small neighbourhood of $p$ (Fig.~\ref{fig:constr}).
\begin{figure}
\centering
\begin{overpic}[scale=0.6]{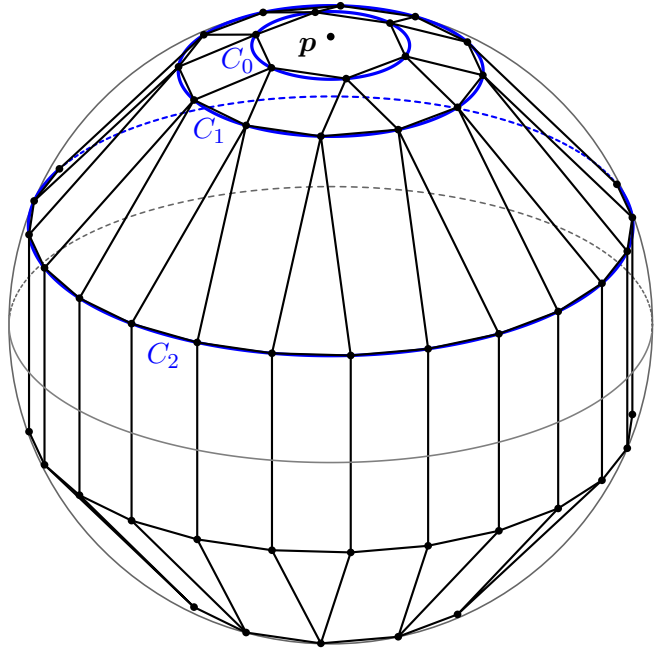}
    \put(45.5,89){$\vp$}
    \put(34.5,86.5){\color{blue}$C_0$}
    \put(30.5,76.5){\color{blue}$C_1$}
    \put(24,44.5){\color{blue}$C_2$}
\end{overpic}
\caption{Construction in Lemma~\ref{construction} ($k=6$, $m=3$)}\label{fig:constr}
\end{figure}
Choose the regular polygons in such a way that an edge can be chosen from each polygon so that all the chosen edges (when considered as chords of the sphere) are parallel.
Let $V$ be the set of all the vertices of these $m$ polygons together with their antipodal points.
Then \[\card{V}=2(k+2k+2^2k+\dots+2^{m-1}k)=2(2^m-1)k.\]
We next count the number of double-normal pairs by first counting the number of faces of the Delaunay tiling.
We only present the case $k>4$.
(The case $k=4$ is exactly the same, but with slightly different notation.)

The faces of the Delaunay tiling of $V$ are, apart from two spherical $k$-gons, spherical triangles and equiangular spherical quadrilaterals.
In the region bounded by $C_i$ and $C_{i+1}$ there are $2^ik$ spherical triangles and $2^ik$ spherical quadrilaterals ($i=0,\dots,m-2$).
In the region between $C_{m-1}$ and $-C_{m-1}$ there are $2^{m-1}k$ spherical quadrilaterals (and no spherical triangles).
Finally, there are $2$ spherical $k$-gons.
In the notation of the proof of Theorem~\ref{lemma:sphere}, the number of triangles is \[f_3=2(k+2k+\dots+2^{m-2}k)=2(2^{m-1}-1)k,\] the number of spherical quadrilaterals is \[f_4=2(k+2k+\dots+2^{m-2}k)+2^{m-1}k=(2^m+2^{m-1}-2)k,\] and the number of $k$-gons is $f_k=2$.
Let $e$ denote the number of edges of the Delaunay triangulation.
By Euler's formula, $\card{V}-e+f_3+f_4+f_k=2$.
It follows that $e=k(2^{m+2}+2^{m-1}-6)$.

Finally, we calculate the number of double-normals.
The edges $\vx\vy$ and $\vx'\vy'$ of the weak Gabriel graph $G=(V,E)$ correspond to the non-antipodal double-normal pairs $\vx\vy'$ and $\vx'\vy$.
There are $\frac12\card{V}$ double-normal antipodal pairs of points.
Therefore,
\begin{align*}
N(V)&=\card{E}+\frac12\card{V}=e+2f_4+\frac{k}{2}f_k+\frac{1}{2}\card{V}\\
&=(2^{m+3}+2^{m-1}-10)k=\frac{17}{4}\card{V}-\frac{3}{2}k.\qedhere
\end{align*}
\end{proof}
The first part of Theorem~\ref{thm:sphereconstruction} follows from Lemma~\ref{construction} if we set $k=4$.
For general values of $n$, we let $k$ and $2^m$ be of the order of $\sqrt{n}$, use the construction of $V$ from Lemma~\ref{construction}, making sure that $\card{V}\leq n$ with $n-\card{V}=O(\sqrt{n})$, and then add the lacking points inside some triangle of the Delaunay tiling.

More precisely, let $n\geq 16$, $m=\lfloor\frac{1}{2}\log_2 n -1\rfloor$, and $k=2\lfloor n/(4(2^{m}-1))\rfloor$, and apply Lemma~\ref{construction}.
The resulting set $V\subset\bS^2$ satisfies \[n-(2^{m+2}-4)<\card{V}=2(2^m-1)k\leq n,\] hence, $n-\card{V}<2^{m+2}\leq 2\sqrt{n}$ and $N(V)=\frac{17}{4}\card{V}-3k/2=\frac{17}{4}n-O(\sqrt{n})$.
If we add $n-\card{V}$ points in the interior of some spherical triangle $\triangle\va\vb\vc$ of the Delaunay tiling of $V$, 
we destroy the $6$ double-normal pairs $\va\vb'$, $\va'\vb$, $\vb\vc'$, $\vb'\vc$, $\va\vc'$, $\va'\vc$, 
while perhaps adding some more double-normal pairs.
We end up with a set of $n$ points with $\frac{17}{4}n-O(\sqrt{n})$ double-normal pairs, which shows the second part of Theorem~\ref{thm:sphereconstruction}.

\subsection*{Acknowledgement}
We thank Endre Makai for a careful reading of the manuscript and for many enlightening comments.

\end{document}